\documentclass[11pt]{amsart}
\theoremstyle{plain}
\newtheorem{thm}{Theorem}[section]
\newtheorem{theorem}[thm]{Theorem}

\newtheorem{lemma}[thm]{Lemma}
\newtheorem{corollary}[thm]{Corollary}
\newtheorem{proposition}[thm]{Proposition}
\theoremstyle{definition}
\newtheorem{remark}[thm]{Remark}

\newtheorem{definition}[thm]{Definition}
\newtheorem{claim}[thm]{Claim}

\numberwithin{equation}{section}
\newcommand{\sO}{{\mathcal O}}

\newcommand{\C}{{\mathbb C}}

\newcommand{\BP}{{\mathbb P}}
\newcommand{\Q}{{\mathbb Q}}
\newcommand{\R}{{\mathbb R}}

\newcommand{\Z}{{\mathbb Z}}

\title [Primitive automorphisms]{A criterion for the primitivity of a birational automorphism of a Calabi-Yau manifold and an application}

\author{Keiji Oguiso}

\address{Mathematical Sciences, the University of Tokyo, Meguro Komaba 3-8-1, Tokyo, Japan and Korea Institute for Advanced Study, Hoegiro 87, Seoul, 
133-722, Korea}
\email{oguiso@ms.u-tokyo.ac.jp}

\thanks{The author is supported by JSPS Grant-in-Aid (S) No 25220701, JSPS Grant-in-Aid (S) 15H05738, JSPS Grant-in-Aid (A) 16H02141, JSPS Grant-in-Aid (B) 15H03611, and by KIAS Scholar Program.}

\begin{document}

\maketitle

\begin{abstract} We shall give a sufficient condition on the primitivity of a birational automorphism of a Calabi-Yau manifold in purely algebro-geometric terms. As an application, we shall give an explicit construction of Calabi-Yau manifolds of Picard number $2$ of any dimension $\ge 3$, with primitive birational automorphisms of first dynamical degree $>1$.
\end{abstract}

\section{Introduction}

Throughout this note, we work in the category of projective varieties defined over $\C$. This note is a continuation of \cite{Og14}, \cite{Og16-2} and is much inspired by recent works of Bianco \cite{Bi16} in technique and also works of Amerik-Verbitsky \cite{AV16} and Ouchi \cite{Ou16} in spirit for dynamical studies of automorphisms of hyperk\"ahler manifolds. Our main interest is the existence of primitive birational automorphisms of the first dynamical degree $>1$ on Calabi-Yau manifolds in any dimension $\ge 3$, especially of the smallest possible Picard number $2$, as done by \cite{AV16}, \cite{Ou16} for hyperk\"ahler manifolds. Our main results are Theorems \ref{thm1}, \ref{thm2} below.

The following notion of primitivity is introduced by De-Qi Zhang \cite{Zh09} and plays the central role in this note:

\begin{definition}\label{def1} Let $X$ be a projective variety and $f \in {\rm Bir}\, (X)$. We call a rational dominant map $\pi : X \dasharrow B$ to a projective variety $B$ with connected fibers a {\it rational fibration}. The map $\pi$ is said to be {\it non-trivial} if $0 < \dim\, B < \dim\, X$. A rational fibration $\pi : X \dasharrow B$ is said to be $f$-{\it equivariant} if there is $f_B \in {\rm Bir}\, (B)$ such that $\pi \circ f = f_B \circ \pi$. We say that $f$ is {\it primitive} if there is no non-trivial $f$-equivariant rational fibration $\pi : X \dasharrow B$. 
\end{definition}

In this definition, if we are interested in the birational equivalence class of $(X, f)$ rather than the variety $X$ itself, then we may assume without loss of generality that $X$ and $B$ are smooth and $\pi$ is a morphism by a Hironaka resolution (\cite{Hi64}). Primitivity of a birational automorphism in the category of projective varieties is an analogue of irreducibility of a linear selfmap in the category of finite dimensional linear spaces.  

The notion of dynamical degrees is introduced by Dinh-Sibony \cite{DS05} as a refinement of the notion of topological entropy in complex dynamics. Though the notion of primitivity is purely algebro-geometric, there observed some close relations between primitivity and dynamical degrees (see eg., \cite{Og16-1} for surfaces and \cite{Bi16} for hyperk\"ahler manifolds). Here we briefly recall the definition of dynamical degrees, following \cite{Tr15}.   

For a rational map $g : X \dasharrow Y$ from a smooth projective variety $X$ to a smooth projective variety $Y$, we define $g^* : N^p(Y) \to N^p(X)$ by 
$$g^* = (p_1)_* \circ p_2^*\,\, .$$ 
Here $N^p(X)$ is a free $\Z$-module generated by the numerical equivalence classes of $p$-cocycles, $p_1 : Z \to X$ is a Hironaka resolution of the indeterminacy of $g$, $p_2 : Z \to Y$ is the induced morphism, $p_2^*$ is the pullback as a cocycle and $(p_1)_*$ is the pushforward as a cycle (\cite[Chap19, Example 19.1.6]{Fu84}). In this definition, smoothness of $Y$ is not needed but smoothness of $Z$ and $X$ are needed in order to identify cycles and 
cocycles over $\Z$. Then, the $p$-th {\it dynamical degree} of $f \in {\rm Bir}\, (X)$ is defined by 
\begin{equation}\label{eq11}
d_p(f) := \lim_{k \to \infty} (((f^k)^*H^p).H^{n-p})_X^{\frac{1}{k}}\,\, .
\end{equation}
Here $n := \dim\, X$ and $H$ is any ample divisor on $X$. By Dinh-Sibony \cite{DS05} (see also \cite{Tr15}), $d_p(f)$ is well-defined, independent of the chioce $H$, and birational invariant in the sense that 
$$d_p(\varphi^{-1} \circ f \circ \varphi) = d_p(f)$$ 
for any birational map $\varphi : X' \dasharrow X$ from a smooth projective variety $X'$. Moreover, if $f \in {\rm Aut}\, (X)$, then $d_p(f)$ coincides with the spectral radius $r_p(f)$ of the linear selfmap $f^*|N^p(X)$ and the topological entropy $h_{\rm top}(f)$ is computed as:
$$h_{\rm top}(f) = {\rm Max}_{p} \log d_p(f) = {\rm Max}_{p} \log r_p(f)\,\, .$$
Throughout this note, we call $X$ a {\it Calabi-Yau manifold}, if $X$ is smooth, simply-connected, $H^0(\Omega_X^j) = 0$ for $0 < j < n := \dim\, X$ and $H^0(\Omega_X^{n}) = \C\omega_X$ for a nowhere vanishing regular $n$-form $\omega_X$. We call $X$ a {\it minimal Calabi-Yau variety} if $X$ has at most $\Q$-factorial terminal singularities, $h^1(\sO_X) = 0$ and $\sO_X(K_X) \simeq \sO_X$. 

Our aim is to prove the following two theorems (Theorems \ref{thm1}, \ref{thm2}):

\begin{theorem}\label{thm1} Let $X$ be a minimal Calabi-Yau variety of dimension $n \ge 3$ of Picard number $\rho(X) \ge 2$ and $f \in {\rm Bir}\, (X)$ such that 
\begin{enumerate}
\item For any movable effective divisor $D$, there are a minimal Calabi-Yau variety $X'$ and a birational map $g : X' \dasharrow X$, both allowing to depend on $D$, such that $D' = g^*D$ is semi-ample on $X'$; and 
\item The action $f^*|N^1(X)_{\Q}$ is irreducible over $\Q$.
\end{enumerate}
Then $f$ is primitive. If in addition that $X$ is smooth, then $d_1(f) > 1$ as well. 
\end{theorem}
See Remark \ref{rem1} for the terminologies and precise definitions in conditions (1) and (2). 

The condition (1) in Theorem \ref{thm1} is automatical if the log minimal model program works in dimension $n$ (then one makes $D'$ nef) and log abundance theorem also holds in dimension $n$ (then $D'$ is semi-ample). In particular, as the log minimal model program works in dimension $3$ (\cite{Sh03}) and log abundance theorem also holds in dimension $3$ (\cite{Ka92}, \cite{KMM94}), we obtain:

\begin{corollary}\label{cor1} Let $X$ be a minimal Calabi-Yau variety of dimension $3$ of Picard number $\rho(X) \ge 2$ and $f \in {\rm Bir}\, (X)$ such that the action $f^*|N^1(X)_{\Q}$ is irreducible over $\Q$. Then $f$ is primitive. 
\end{corollary}

Theorem \ref{thm1} and Corollary \ref{cor1} may have their own interest, suggest some relation between primitivity of birational automorphisms and primitivity ($=$ irreducibility) of its geometrically meaningful linear representation. We prove Theorem \ref{thm1} in Section 2. 

\begin{theorem}\label{thm2} For each $n \ge 3$, there is an $n$-dimensional Calabi-Yau manifold $M$ of Picard number $\rho(M) = 2$ with a primitive birational automorphism $f \in {\rm Bir}\, (M)$ of the first dynamical degree $d_1(f) > 1$. 
\end{theorem}

Our manifold $M$ and $f \in {\rm Bir}\, (M)$ in Theorem \ref{thm2} are explicit (See Section 3). If $f$ is primitive, then ${\rm ord}\, (f) = \infty$ (\cite[Lemma 3.2]{Og16-1}). In particular, $\rho(M) \ge 2$ if $M$ is a Calabi-Yau manifold with primitive $f \in {\rm Bir}\, (M)$ (see eg. \cite{Og14}). So, Theorem \ref{thm2} also shows that the estimate $\rho(M) \ge 2$ is optimal for Calabi-Yau manifolds with primitive birational automorphisms in each dimension $\ge 3$. We prove Theorem \ref{thm2} in Section 3 as an application of Theorem \ref{thm1}.

We refer \cite{Og15}, \cite{Og16-2} and references therein for background and known results relevant to our main theorems. The following standard terminologies and remarks will be frequently used in this note: 

\begin{remark}\label{rem1} Let $X$ be a normal projective variety and $f \in {\rm Bir}\, (X)$.

\begin{enumerate}

\item Following \cite{Ka88}, we say that a $\Q$-Cartier Weil divisor $D$ is {\it movable} (resp. {\it semi-ample}) if there is a positive integer $m$ such that the complete linear system $|mD|$ has no fixed component (resp. is free).

\item Assume that $X$ is not necessarily smooth but $\Q$-factorial. Then any Weil divisors (codimension one cycles) are $\Q$-Cartier divisors (codimension one cocycles) on $X$. We denote by $N^1(X)_{\Q}$ the finite dimensional $\Q$-linear space spanned by the numerical equivalence classes of Weil divisors on $X$ over $\Q$. So, in this case, one can define the map $g^* : N^1(Y)_{\Q} \to N^1(X)_{\Q}$ for a rational map $g : X \dasharrow Y$ in the same way as explained already (over $\Q$, not necessarily over $\Z$). 

\item Assume that either $f \in {\rm Aut}\, (X)$ or $X$ is $\Q$-factorial and $f \in {\rm Bir}\, (X)$ is isomorphic in codimension one. For instance, this is the case when $X$ is a minimal Calabi-Yau variety and $f \in {\rm Bir}\, (X)$ (see eg. \cite[Page 420]{Ka08}). Then, $d_1(f)$ is well-defined by the same formula (\ref{eq11}). Moreover, the correspondence $f \mapsto f^*$ is functorial in the sense that $(f \circ g)^* = g^* \circ f^*$ for $f, g \in {\rm Aut}\, (X)$ in the first case and for any $f, g \in {\rm Bir}\, (X)$ in the second case. In particular, $f^* \in {\rm GL}\, (N^1(X)_{\Q})$. Moreover, $d_1(f)$ coincides with the spectral radius of $f^* | N^1(X)_{\Q}$ by $(f^k)^* = (f^*)^k$ for any $k \in \Z$. One observes this by combining the Birkhoff-Perron-Frobenius theorem (\cite{Br67}) applied for the movable cone $\overline{\rm Mov}\,(X) \subset N^1(X)_{\R}$ and the Jordan canonical form of $f^*|N^1(X)_{\Q}$ over $\C$. 

\end{enumerate} 
\end{remark}

\section{Proof of Theorem \ref{thm1}.}

In this section, we prove Theorem \ref{thm1}. Throughout this section, unless stated otherwise, $X$ is a minimal Calabi-Yau variety satisfying the conditions (1) and (2) in Theorem \ref{thm1}. 

\begin{lemma}\label{lem21} $X$ has no $f$-equivariant rational fibration $\pi : X \dasharrow B$ such that $0 < \dim\, B < \dim\, X$ and $\kappa(\tilde{X}_b) = 0$ for general $b \in B$. Here $\nu : \tilde{X} \to X$ is a Hironaka resolution of the indeterminacy of $\pi$ and the singularities of $X$, $\tilde{X}_b$ is the fiber over $b \in B$ of the morphism $\tilde{\pi} := \pi \circ \nu : \tilde{X} \to B$ and $\kappa(\tilde{X}_b)$ is the Kodaira dimension of $\tilde{X}_b$. 
\end{lemma}

\begin{proof} Assuming to the contrary that $X$ has an $f$-equivariant rational fibration $\pi : X \dasharrow B$ such that $0 < \dim\, B < \dim\, X$ and $\kappa(\tilde{X}_b) = 0$ for general $b \in B$, we shall derive 
a contradiction.

By taking a Hironaka resolution of singularities, we may and will assume that $B$ is smooth. Let $H$ be a very ample divisor on $B$. Set $L := \nu_*(\tilde{\pi}^*H)$. Then $|L|$ is movable and $\pi = \Phi_{\Lambda}$. Here $\Lambda \subset |L|$ is a sublinear system of $|L|$ and $\Phi$ is the rational map associated to $\Lambda$. Then, by the assumption (1) in Theorem \ref{thm1}, there are a minimal Calabi-Yau variety $X'$ and a birational map $g : X' \dasharrow X$ such that $g^*L$ is semi-ample. So, by replacing $(X, f)$ by 
$$(X', f' := g^{-1} \circ f \circ g)\,\, ,$$ 
we may and will assume that $|L|$ is semi-ample. Note that $(f')^*| N^1(X')_{\Q}$ is irreducible over $\Q$, as $g^{\pm 1}$ and $f$ are isomorphic in codimension one so that 
$$(f')^* = (g^{-1} \circ f \circ g)^* = g^* \circ f^* \circ (g^*)^{-1}\,\, .$$
Let us take a sufficiently large positive integer $m$ such that $|mL|$ is free and the morphism $\varphi_m = \Phi_{|mL|}$ is the Iitaka-Kodaira fibration associated to $L$. Set $B' = \varphi(X)$. Then $B'$ is normal, projective and $\dim\, B' = \kappa(L, X)$. Here $\kappa(L, X)$ is the Iitaka-Kodaira dimension of $L$. As $\Lambda \subset |L| \subset |mL|$, the surjective morphism $\varphi_m : X \to B'$ factors through $\pi : X \dasharrow B$, i.e., there is a (necessarily dominant) rational map $\rho : B' \dasharrow B$ such that $\pi = \rho \circ \varphi_m$. 

\begin{claim}\label{cl22}
$\rho : B' \dasharrow B$ is a birational map.
\end{claim}
\begin{proof}
As $\pi$ is of connected fibers and $\rho$ is dominant, it suffices to show that $\dim\, B = \dim\, B'$, i.e., that $\dim\, B =  \kappa(L, X)$. Set $\tilde{L} := \nu^*L$. Then $\kappa(L, X) = \kappa(\tilde{L}, \tilde{X})$. In what follows, we shall prove $\dim\, B = \kappa(\tilde{L}, \tilde{X})$. 

Let $\{E_i\}_{i=1}^{m}$ be the set of exceptional prime divisors of $\nu$. Then we can write
$$\nu^*L = \tilde{\pi}^*H + \sum_{i=1}^{m} a_iE_i\,\, ,\,\, K_{\tilde{X}} = \sum_{i=1}^{m} b_iE_i\,\, .$$
Here $a_i \ge 0$ as $\nu$ resolves the indeterminacy of $\pi$ and $b_i > 0$ as $X$ is a minimal Calabi-Yau variety. 
Note that $\tilde{L}|_{\tilde{X}_b} = a_iE_i|_{\tilde{X}_b}$ as $H|_{\tilde{X}_b}$ is trivial and $K_{\tilde{X}_b} = b_i E_i|_{\tilde{X}_b}$ by the adjunction formula. Then there is a positive integer $m$ such that $m K_{\tilde{X}_b} - \tilde{L}|_{\tilde{X}_b}$ is linearly equivalent to an effective divisor. 
Hence
$$0 \le \kappa(\tilde{L}|_{\tilde{X}_b}, \tilde{X}_b) \le \kappa(\tilde{X}_b) = 0\,\, ,$$
as $b \in B$ is general and $\kappa(\tilde{X}_b) = 0$ by the assumption. Thus 
$\kappa(\tilde{L}|_{\tilde{X}_b}, \tilde{X}_b) = 0$. Therefore
$$\dim\, B  \le \kappa(\tilde{L}, \tilde{X}) \le \dim\, B + \kappa(\tilde{L}|_{\tilde{X}_b}, \tilde{X}_b) = \dim\, B\,\, .$$
Here the second inequality follows from \cite[Theorem 5.11]{Ue75}. 
Hence $\dim\, B = \kappa(\tilde{L}, \tilde{X})$ as desired.  
\end{proof}

By Claim \ref{cl22}, we may and will assume that $\pi : X \to B$ is an $f$-equivariant surjective {\it morphism} given by the free complete linear system $|L|$, by replacing $\pi : X \dasharrow B$ by $\varphi_m : X \to B'$ for sufficiently large divisible $m$. Then $B = {\rm Proj}\,\oplus_{k\ge 0} H^0(X, \sO_X(kL))$ and $\pi = \Phi_{|L|}$. 
\begin{claim}\label{cl23}
There is $f_B \in {\rm Aut}\, (B)$ ({\it not only in ${\rm Bir}\, (B)$}) such that $f_B \circ \pi = \pi \circ f$ as a rational map from $X$ to $B$. 
\end{claim}
\begin{proof}
As $f$ is isomorphic in codimension one, the pullback $f^*$ induces an isomorphism 
$$f_* : B = {\rm Proj}\,\oplus_{k\ge 0} H^0(X, \sO_X(kL)) \simeq {\rm Proj}\, \oplus_{k\ge 0} H^0(X, \sO_X(kf^*L)) = B\,\, .$$
Here the last equality is the one under the identification of ${\rm Proj}\,\oplus_{k\ge 0} H^0(X, \sO_X(kL))$ with ${\rm Proj}\,\oplus_{k\ge 0} H^0(X, \sO_X(kf^*L))$ defined by $D \mapsto f^*D$ for $D \in |kL|$. 

Then the image of $\pi \circ f = \Phi_{|f^*L|}$ is $B = {\rm Proj}\,\oplus_{k\ge 0} H^0(X, \sO_X(kf^*L))$ under the last identification made above. Thus 
$$f_* \in {\rm Aut}\, (B)\,\, $$ 
under the first identification $B = {\rm Proj}\,\oplus_{k\ge 0} H^0(X, \sO_X(kL))$ and $f_{*}$ satisfies that $\pi \circ f = f_* \circ \pi$. We may now take $f_B = f_*$.  
\end{proof}
As $f_B \in {\rm Aut}\, (B)$, the map $f_B^* : N^1(B)_{\Q} \to N^1(B)_{\Q}$ is a well-defined isomorphism (cf. Remark \ref{rem1}). Note also that 
$$\pi^* \circ f_B^* = f^* \circ \pi^*\,\, .$$ 
Indeed, $(f_B \circ \pi)^* = \pi^* \circ f_B^*$, as $\pi$ and $f_B$ are morphisms. We have also $(\pi \circ f)^* = f^* \circ \pi^*$, as $f$ is isomorphic in codimension one and $X$ is normal and $\Q$-factorial (cf. Remark \ref{rem1}). Thus $\pi^* \circ f_B^* = f^* \circ \pi^*$ from $f_B \circ \pi = \pi \circ f$.  Hence the subspace $\pi^*N^1(B)_{\Q} \subset N^1(X)_{\Q}$ is $f$-stable by $f_B^*( N^1(B)_{\Q}) = N^1(B)_{\Q}$. On the other hand, as $0 < \dim\, B < \dim\, X$, $X$ and $B$ are projective and $\pi$ is a {\it morphism}, 
it follows that 
$$0 \not= \pi^*N^1(B)_{\Q} \not= N^1(X)_{\Q}\,\, ,$$ 
a contradiction to the irreducibility of the action $f^*$ on $N^1(X)_{\Q}$, i.e., the assumption (2) in Theorem \ref{thm1}. 
This completes the proof of Lemma \ref{lem21}.  
\end{proof}

We say that a statement (P) on closed points of a projective variety $B$ holds for {\it any very general point} $b \in B$ if there is a countable union $Z$ of proper Zariski closed subsets of $B$ such that (P) holds for any $b \in B \setminus Z$.

\begin{lemma}\label{lem22} Let $P$ be a very general closed point of $X$. Then $f^n$ is defined at $P$ for all $n \in \Z$ and the set 
$\{f^n(P)\, |\, n \in \Z\}$ is Zariski dense in $X$. 
\end{lemma}

\begin{proof} The first assertion is clear. We show the second assertion. By \cite[Th\'eor\`eme 4.1]{AC13}, there is a smooth projective variety $C$ and a dominant rational map $\rho : X \dasharrow C$ such that $\rho \circ f = \rho$ as a rational map and $\rho^{-1}(\rho(P))$ is the Zariski closure of $\{f^n(P)\, |\, n \in \Z\}$ for very general $P \in X$. 

Recall that the map $\rho^*  : N^1(C)_{\Q} \to N^1(X)_{\Q}$ is well-defined, as $X$ is normal and $\Q$-factorial (Remark \ref{rem1}). As $f$ is isomorphic in codimension one and $\rho = \rho \circ f$, we have 
$$\rho^* = f^* \circ \rho^* : N^1(C)_{\Q} \to N^1(X)_{\Q}$$ 
as in the proof of Lemma \ref{lem21} (See also \cite[Lemma 4.5]{Bi16}). Thus 
$$f^*|\rho^*N^1(C)_{\Q} = id_{\rho^*N^1(C)_{\Q}}\,\, .$$ 
As $\rho(X) \ge 2$ and $f$ is irreducible on $N^1(X)_{\Q}$ by our assumptions in Theorem \ref{thm1}, it follows that 
$$\rho^*N^1(C)_{\Q} = \{0\}\,\, .$$ 
As $C$ is projective, this is possible only when $\dim\, C = 0$, i.e., $C$ is a point. This implies the second assertion. 
\end{proof}

The following important proposition is due to Bianco \cite[Proposition]{Bi16}: 
\begin{proposition}\label{prop23} Let $X$ be a projective variety (not necessarily a minimal Calabi-Yau variety) and $f \in {\rm Bir}\, (X)$. Assume that 
$\pi : X \dasharrow B$ is a nontrivial $f$-equivariant rational fibration such that $\tilde{X}_b$ is of general type for general $b \in B$. Here $\nu : \tilde{X} \to X$ is a Hironaka resolution of the singularities of $X$ and the indeterminacy of $\pi$ and $\tilde{X}_b$ is the fiber over $b \in B$ of the morphism $\tilde{\pi} := \pi \circ \nu : \tilde{X} \to B$. Then for any very general point $x \in X$, the well-defined $\langle f \rangle$-orbit $\{f^n(x)| n \in \Z\}$ is never Zariski dense in $X$.  
\end{proposition}

\begin{remark}\label{rem24} The assertion and proof of \cite[the second statement of Proposition]{Bi16} seem a bit too optimistic (as isotriviality does not necessarily imply global triviality after global \'etale covering. Indeed most minimal ruled surfaces do not admit any global trivializations) and the proof of \cite[the third statement of Proposition]{Bi16}, which we cited above, is based on it. We shall give a more direct proof of Proposition \ref{prop23} below, which is a slight modification of the original proof of \cite[the first statement of Proposition]{Bi16}. 
\end{remark}

\begin{proof} By taking a Hironaka resolution of $B$, $X$ and the indeterminacy $X \dasharrow B$, we may and will assume that $\pi : X \to B$ is a surjective morphism between smooth projective varieties. Let $b \in B$ be any very general point. We may assume without loss of generality that $f_B^n$ ($n \in \Z$) are defined at $b$. We set 
$$O(b) := \{f_B^n(b)| n \in \Z\}\,\, .$$ 
We may and will assume that $O(b)$ is Zariski dense in $B$ for any very general point $b \in B$, as otherwise, the assertion is obvious. Note that, for a given Zariski dense open subset $U' \subset B$, the set $O(b) \cap U'$ is Zariski dense in $B$ and we may also assume that $b \in U'$, as the assertion is made for very general $b \in B$ so that one can remove $B \setminus U'$. This convention will be employed in the rest of proof, whenever it will be convenient. 

Take a Zariski dense open subset $U \subset B$ such that $\pi_{U} := \pi|\pi^{-1}(U) : \pi^{-1}(U) \to U$ is a smooth morphism. As remarked above, we may and will assume that $b \in U$. Then take the relative canonical model 
$$p := \pi_{U}^{\rm can} : Y := {\rm Proj} \oplus_{m \ge 0} (\pi_{U})_* \sO_{X}(mK_{\pi^{-1}(U)})\to U$$  
of $\pi_U : \pi^{-1}(U) \to U$ over $U$ (see \cite[Corollary 1.1.2]{BCHM10}, \cite[Theorem 6.6]{Ka09} for the finite generation). We denote the fiber of $p$ (resp. of $\pi$) over $t \in U$ by $Y_t$ (resp. $X_t$). Note that two canonical models are birationally isomorphic if and only if they are isomorphic (\cite[Corollary 14.3 and its proof]{Ue75}). This is the advantage to pass to the canonical models for us. 

By definition of $Y$, there is a positive integer $\ell$ such that $L := \sO_Y(\ell)$ is a $p$-very ample line bundle with $L|_{Y_t} = \sO_{Y_t}(\ell) = \sO_{Y_t}(\ell K_{Y_t})$ for all $t \in U$. The Euler characteristic $\chi(\sO_{Y_t}(m \ell))$ is constant as a function of $t \in U$ for any large integer $m$ by the invariance of the pluri-genera (\cite{Si98}). Indeed, as $Y_t$ is a canonical model of $X_t$, we have
$$h^0(Y_t, \sO_{Y_t}(m \ell)) = h^0(Y_t, \sO_{Y_t}(m \ell K_{Y_t})) = h^0(X_t, \sO_{X_t}(m \ell K_{X_t}))\,\, .$$
The last term is constant by the invariance of the pluri-genera (\cite{Si98}). We have also $h^i(Y_t, \sO_{Y_t}(m \ell)) = 0$ for all $i > 0$ and for all large $m$ by the Serre vanishing theorem. 

Thus, the projective morphism $p : Y \to U$ is flat. {\it From now we fix $L$ and $\ell$ above and regard $f \in {\rm Bir}\, (Y)$ via the pluri-canonical map $X \dasharrow Y$.} 
 
Let $F = Y_b$ be the fiber of $p$ over $b$. Then the second projection $q : F \times U \to U$ is also a flat projective morphism. Set $H := r^*\sO_{F}(\ell K_F)$. Here $r : F \times U \to F$ is the first projection. Then $H$ is a $q$-very ample line bundle on $F \times U$. We denote by $M$ the relatively very ample line bundle of $(F \times U) \times_U Y \to U$ given by the tensor product of the pull back of $L$ and $H$. 

Now consider the relative isomorphism functor ${\it Isom}_U(F \times U, Y)$ over $U$, which associates to any $U$-scheme $V$ the set ${\it Isom}_U(F \times U, Y)(V)$ of isomorphisms from $(F \times U) \times_U V = F \times V$ to $Y \times_{U} V$ over $V$. As $p$ and $q$ are both projective and flat, the functor ${\it Isom}_U(F \times U, Y)$ is represented by a $U$-scheme ${\rm Isom}_U(F \times U, Y)$, which is realized as an open subscheme of the relative Hilbert scheme ${\rm Hilb}_U((F \times U) \times_U Y)$ over $U$ under the identification of an isomorphism and its graph (see for instance \cite[Theorem 5.23 and its proof]{Ni05}). We denote the structure morphism ${\rm Isom}_U(F \times U, Y) \to U$ 
by $\tau$. 

Let $V \subset U$ be any connected subscheme of $U$. Then, as $F$ is a canonical model of general type, the group ${\it Aut}_{U}(F \times U)(V) := {\it Isom}_{U}(F \times U, F \times U)(V)$ of relative automorphisms over $V$ is isomorphic to ${\rm Aut}(F)$ under $g \mapsto g \times id_V$. This is because any morphism $V \to {\rm Aut}(F)$ is constant, as ${\rm Aut}(F)$ is a finite group (\cite[Corollary 14.3 and its proof]{Ue75}) and $V$ is connected. Moreover, the set ${\it Isom}_U(F \times U, Y)(V)$ is linear, in the sense that any element is given by a projective linear isomorphism under the embedding relative over $V$, with respect to $L|_{p^{-1}(V)}$ and $H|_{q^{-1}(V)}$. This is because the pluri-canonical linear system is preserved under isomorphisms. Hence the Hilbert polynomial of the graph of any isomorphism $\varphi_t : F \to Y_t$ with respect to $M$ is independent of $t \in U$ and $\varphi_t$. We denote the polynomial by $P$. 

Then, under the open embedding ${\rm Isom}_U(F \times U, Y) \subset {\rm Hilb}_U((F \times U) \times_U Y)$, we have
$${\rm Isom}_U(F \times U, Y) = {\rm Isom}_U^{P}(F \times U, Y) \subset {\rm Hilb}_U^{P}((F \times U) \times_U Y)\,\, .$$ 
As the last inclusion is an open immersion and ${\rm Hilb}_U^{P}((F \times U) \times_U Y)$ is projective over $U$ with only finitely many irreducible 
components, it follows that ${\rm Isom}_U(F \times U, Y)$ is quasi-projective over $U$ with only finitely many irreducible components. Recall that $b \in O(b) \cap U$ is Zariski dense in $U$. As $X_{f_B^n(b)}$ and $X_b$ are birational, their canonical models $Y_{f_B^n(b)}$ and $Y_b = F$ are isomorphic. Thus 
$$O(b) \cap U \subset \tau({\rm Isom}_U(F \times U, Y)) \subset U\,\, .$$
Here $\tau({\rm Isom}_U(F \times U, Y))$ is a constructible subset and has only finitely many irreducible components as well. Then, at least one of the irreducible components of $\tau({\rm Isom}_U(F \times U, Y))$, say $V$, is Zariski dense in $U$, as so is $O(b) \cap U$. As $V$ is constructible, it follows that there is a Zariski dense open subset 
$W \subset U$ 
such that $W \subset V \subset \tau({\rm Isom}_U(F \times U, Y))$. Hence possibly after shrinking $W$ a bit, we have ${\it Isom}_U(F \times U, Y)(W) \not= \emptyset$ and therefore $F \times W$ is isomorphic to $Y|_{W}$ over $W$. Now choosing an isomorphism $\rho : F \times W \to Y|_W$ over $W$, we identify $F \times W = Y|_W$ and regard $f \in {\rm Bir}\, (F \times W)$. We denote $f_W = f_B|W \in {\rm Bir}\, (W)$. Then the morphism $q : F \times W \to W$ is $f$-equivariant. 

Let $y := (t, s) \in F \times W$ be any very general point. As $F$ is a canonical model, the map $f^n|F \times \{s\} : F \times \{s\} \to F \times \{f_W^n(s)\}$ is an isomorphism whenever $f_W^n(s) \in W$. In particular, 
$$f^n(y) = f^n((t, s)) = (g_{s}(t), f_W^n(s))$$
for some $g_{s} \in {\rm Aut}\, (F)$. Thus, 
$$\{f^n(y)\, |\, n \in \Z\} \cap (F \times W)\, \subset\, ({\rm Aut}\, (F) \cdot t) \times W\,\, ,$$ 
the latter of which is a proper closed subset of $F \times W$ as $|{\rm Aut}\, (F)| < \infty$. Therefore, the orbit $\{f^n(y)\, |\, n \in \Z\} \cap (F \times W)$ is not Zariski dense in $F \times W$. As $X$ and $F \times W$ are birational under an $f$-equivariant map, this implies the result. 
\end{proof}

Now we are ready to complete the proof of Theorem \ref{thm1}. 

Assume that $X$ admits a non-trivial $f$-equivariant rational fibration $\pi : X \dasharrow B$. Let $\tilde{\pi} : \tilde{X} \to B$ be a Hironaka resolution of indeterminacy of $\pi$ and ${\rm Sing}\, (X)$. Consider the relative Kodaira fibration over $B$ (see \cite[Corollary 1.1.2]{BCHM10}, \cite[Theorem 6.6]{Ka09} for the finite generation):
$$g : X \dasharrow K := {\rm Proj} \oplus_{m \ge 0}\tilde{\pi}_* \sO_{\tilde{X}}(mK_{\tilde{X}})\,\, .$$ 
Then $g$ is $f$-equivariant and $\kappa(\tilde{X}_k) = 0$ for general $k \in K$. Here $\tilde{X}_k$ is a Hironaka resolution of the fiber over $k \in K$. By Lemma \ref{lem22} and Proposition \ref{prop23} and by $\dim\, B > 0$, we have $0< \dim\, K < \dim\, X$. However, this contradicts to Lemma \ref{lem21}. Thus $f$ is primitive. 

We shall show that $d_1(f) > 1$ if $X$ is smooth. As $f$ is isomorphic in codimension one, $f^*$ preserves the movable cone $\overline{\rm Mov}(X) \subset N^1(X)_{\R}$, which is, by definition, the closed convex hull of the movable divisor classes in $N^1(X)_{\R}$. As $\overline{\rm Mov}(X)$ is a strictly convex closed cone in $N^1(X)_{\R}$, it follows from the Birkhoff-Perron-Frobenius theorem (\cite{Br67}) 
that there is $0 \not= v \in \overline{\rm Mov}(X)$ such that $f^*v = d_1(f)v$. As $X$ is smooth (hence the Weil divisors are Cartier divisors), $f^* : N^1(X) \to N^{1}(X)$ is a well-defined isomorphism over $\Z$ (not only over $\Q$). Thus the product of the eigenvalues of $f^*$ is of absolute value $1$, as $\Z^{\times} = \{\pm 1\}$. Thus $d_1(f) \ge 1$. If $d_1(f) = 1$, then there would be $0 \not= u \in N^1(X)_{\Q}$ such that $f^*(u) = u$. Hence $d_1(f) >1$, as $f^*|N^1(X)_{\Q}$ is irreducible over $\Q$ and $\rho(X) \ge 2$ by our assumptions in Theorem \ref{thm1}. This completes the proof of Theorem \ref{thm1}. 

\section{Proof of Theorem \ref{thm2}.}

In this section, we shall prove Theorem \ref{thm31} below, from which Theorem \ref{thm2} follows. 

Calabi-Yau manifolds in Theorem \ref{thm31} are higher dimensional generalization of Calabi-Yau threefolds studied in \cite[Section 6]{Og14}.  

Let $n$ be an integer such that $n \ge 3$. Let 
$$M = F_1 \cap F_2 \cap \ldots \cap F_{n-1} \cap Q \subset \BP^n \times \BP^n$$ be a general complete intersection of $n-1$ hypersurfaces $F_i$ ($1 \le i \le n-1$) of bidegree $(1,1)$ and a hypersurface $Q$ of bidegree $(2, 2)$ in 
$\BP^n \times \BP^n$. Then, by the Lefschetz hyperplane section theorem, $M$ is a smooth Calabi-Yau manifold of dimension $n$ and of Picard number $2$. More precisely, 
$${\rm Pic}\, (M) \simeq N^1(M) = \Z h_1 \oplus \Z h_2\,\, .$$
Here and hereafter, $p_i :  \BP^n \times \BP^n \to \BP^n$ is the projection to the $i$-th factor, $L_i$ is the hyperplane class of the $i$-th $\BP^n$, $H_i := p_i^*L_i$, $h_i = H_i|_M$ ($i=1$, $2$). 

Let 
$$V = \cap_{i=1}^{n-1} F_i \subset \BP^n \times \BP^n\,\, .$$ 
Then $V$ is a smooth Fano manifold of dimension $n+1$ with ${\rm Pic}\, (V) = \Z H_1|V \oplus \Z H_2|V$ and $M \in |-K_V|$. Note that $H_i|V$ gives the $i$-th projection $p_i|V : V \to \BP^n$ to the $i$-th factor. In particular, $H_i|V$ is free (hence nef) but not ample. Then, by a result of Koll\'ar \cite[Appendix]{Bo91}, we have $\overline{{\rm Amp}}\, (V) \simeq \overline{{\rm Amp}}\, (M)$ under the inclusion map $M \subset V$, and therefore
$$\overline{{\rm Amp}}\, (M) = \R_{\ge 0}h_1 + \R_{\ge 0}h_2$$ 
in $N^1(M)_{\R} := N^1(M) \otimes_{\Z} {\R}$. In particular, any nef divisor on $M$ is semi-ample. 

We also note that $g^{*}(h_1 + h_2) = h_1 + h_2$ for $g \in {\rm Aut}\, (M)$. Hence $|{\rm Aut}\, (M)| < \infty$ (see eg. \cite[Proposition 2.4]{Og14}). So, if $f \in {\rm Bir}\, (M)$ is primitive, then necessarily $f \not\in {\rm Aut}\, (M)$ in our case, as $f$ has to be 
of infinite order (\cite[Lemma 3.2]{Og16-1}).  

Consider the projections 
$$\pi_i := p_i|M : M \to \BP^n$$
($i = 1$, $2$). Then $\pi_i$ are of degree $2$ by the definition of $M$. Hence we have a birational involution $\tau_i \in {\rm Bir}\, (M)$ ($i =1$, $2$) corresponding to $p_i$. We consider the birational automorphism $f \in {\rm Bir}\, (M)$ 
defined by 
$$f := \tau_1 \circ \tau_2\,\, .$$
Our main result of this section is the following:

\begin{theorem} \label{thm31}
Under the notation above, $f$ is a primitive birational automorphism of 
$M$ with 
$$d_1(f) = (2n^2 -1) + 2n\sqrt{n^2 -1} >1\,\, .$$
\end{theorem}

\begin{proof}
In the proof, we will frequently use the fact that any birational automorphism of $M$ is isomorphic in codimension one (Remark \ref{rem1}). 

\begin{lemma}\label{lem31}
With respect to the basis $\langle h_1, h_2 \rangle$, the actions of $\tau_i^*|N^1(M)$ ($i=1$, $2$) and $f^*|N^1(X)$ are represented by the following matrices $M_i$, $M_2M_1$ respectively: 
$$M_1 = \left(\begin{array}{rr}
1 & 2n\\
0 & -1\\
\end{array} \right)\,\, ,\,\, M_2 = \left(\begin{array}{rr}
-1 & 0\\
2n & 1\\
\end{array} \right)\,\, ,\,\, M_2M_1 = \left(\begin{array}{rr}
-1 & -2n\\
2n & 4n^2 -1\\
\end{array} \right).$$ 
The eigenvalues of $f^*$ are
$$(2n^2 -1) \pm 2n\sqrt{n^2 -1}\,\, ,$$
which are irrational. In particular, 
$f^*|N^1(X)_{\Q}$ is irreducible over $\Q$.
\end{lemma}

\begin{proof} We have $\tau_1^{*}h_1 = h_1$. We can write $(\pi_1)_*h_2 = aL_1$ for some integer $a$. First, we detemine the value $a$. As $h_2 = \pi_2^*L_2$, we readily compute that
$$a = (aL_1.L^{n-1})_{\BP^n} = ((\pi_1)_*\pi_2^*L_2.L_1^{n-1})_{\BP^n} = (\pi_2^*L_2.\pi_1^*L_1^{n-1})_{M}$$ 
$$= (H_2.H_1^{n-1}.2(H_1+H_2)^n)_{\BP^n \times \BP^n} = 2n\,\, .$$
Thus 
$$h_2 + \tau_1^{*}h_2 = p_1^{*}(p_1)_*h_2 = 2nh_1\,\, .$$
Hence $\tau_1^{*}h_1 = h_1$ and $\tau_2^*h_2 = 2nh_1 - h_2$, and therefore, the matrix representaion $M_1$ of $\tau_1^*$ is as described. In the same way, one obtains the matrix representation $M_2$ of $\tau_2^*$ as described. The matrix representaion of $f^*$ is then $M_2M_1$. The rest follows from a simple computaion of $2 \times 2$ matrices and an elementary fact that $\sqrt{n^2 -1}$ is irrational for any integer $n \ge 2$.
\end{proof}

Let $\tilde{\pi}_i : M \to M_i$ ($i=1$, $2$) be the Stein factorization of $\pi_i : M \to \BP^n$. As the Stein factorization is unique, the covering involution $\tilde{\tau}_i$ of $M_i \to \BP^n$ is in ${\rm Aut}\, (M_i)$ (not only in ${\rm Bir}\, (M_i)$) and satisfies $\tilde{\pi}_i \circ \tau_i = \tilde{\tau}_i \circ \tilde{\pi}_i$. 

\begin{lemma}\label{lem32}
$\tilde{\pi}_i$ ($i=1$, $2$) are small contractions of $M$ and $M$ admits no other contraction. Here a contraction of $M$ means a non-isomorphic surjective morphism to a normal projective variety of positive dimension 
with connected fibers. 
\end{lemma}

\begin{proof} The morphism $\tilde{\pi}_i$ is given by $|mh_i|$ for large $m$. Recall that $h_i$ are not ample. Thus $\tilde{\pi}_i$ is a contraction. As $\rho(M) = 2$, there is then no contraction other than $\tilde{\pi}_i$ ($i=1$, $2$). 

As $\rho(V) = 2$, the $i$-th projection $p_i|V : V \to \BP^n$ from $V$ ($i=1$, $2$) contracts no divisor to a subvariety of codimension $\ge 2$. Thus, the $i$-th projection $\pi_i = p_i|M : M \to \BP^n$ contracts no divisor, as $M = V \cap Q$ and $Q$ is a general very ample divisor on $V$. Hence $\tilde{\pi}_i$ is a small contraction. The fact that $\tilde{\pi}_i$ ($i=1$, $2$) are small contractions of $M$ also follows from the proof of the next Lemma \ref{lem33}. 
\end{proof}

\begin{lemma}\label{lem33} 
${\rm Bir}\, (M) = {\rm Aut}\, (M) \cdot \langle \tau_1, \tau_2 \rangle$.  
\end{lemma}

\begin{proof} As $\rho(M) = 2$, the relative Picard number $\rho(M/M_1)$ 
is $1$. By Lemma (\ref{lem31}), we have
$$\tau_1^*h_2 = -h_2 + 2nh_1\,\, .$$
Thus $\tau_1^*h_2$ is relatively anti-ample for $\overline{\tau}_1^{-1} \circ \tilde{\pi}_1 : M \to M_1$, 
while $h_2$ is relatively ample for $\tilde{\pi}_1 : M \to M_1$. 
Since $K_M = 0$, the map $\tilde{\tau}_1^{-1} \circ \tilde{\pi}_1 : M \to M_1$ is then the flop of $\tilde{\pi}_1 : X \to M_1$, given by $\tau_1$. For the same reason, the map $\tilde{\tau}_2^{-1} \circ \tilde{\pi}_2 : M \to M_2$ is the flop of $\tilde{\pi}_2 : M \to M_2$, given by $\tau_2$. 

Recall that any flopping contraction of a Calabi-Yau manifold is given by a codimension one face of $\overline{{\rm Amp}}\, (M)$ up to automorphisms of $M$ 
(\cite[Theorem 5.7]{Ka88}). As there is no codimension one face of $\overline{{\rm Amp}}\, (M)$ other than ${\mathbf R}_{\ge 0}h_i$ ($i = 1, 2$), there is then no flop of $M$ other than $\tau_i$ ($i = 1, 2$) up to ${\rm Aut}\,(M)$. Recall a fundamental result of Kawamata (\cite[Theorem 1]{Ka08}) that {\it any birational map between minimal models is decomposed into finitely many flops up to automorphisms of the target variety}. Thus any $\varphi \in {\rm Bir}\, (M)$ is decomposed into a finite sequence of flops $\tau_i$ 
and an automorphism of $M$ at the last stage. This proves the result. 
\end{proof}
Set $v_{+} := -h_1 + (n+ \sqrt{n^2-1})h_2$, $v_{-} := -h_2 + (n+ \sqrt{n^2-1})h_1$ and 
$$V := \R_{\ge 0}v_{+} + \R_{\ge 0}v_{-} \subset N^1(M)_{\R}\,\, .$$ 
Here $v_{\pm}$ are eigenvectors of $f^*$ corresponding to the eigenvalues $(2n^2 -1) \pm 2n\sqrt{n^2 -1}$. By writing the Jordan canonical form of the matrix $M_2M_1$ in Lemma \ref{lem31}, one readily observes that 
\begin{equation}\label{eq1}
\lim_{n \to \infty} \R_{>0} (f^n)^* x = \R v_{+}\,\, ,\,\, \lim_{n \to \infty} \R_{>0} (f^{-n})^* x = \R v_{-}
\end{equation}
for any $x \in \overline{{\rm Amp}}\, (M) \setminus \{0\}$. 
\begin{lemma}\label{lem34} 
$\overline{{\rm Mov}}\, (M) = V$. 
Moreover, the interior $(\overline{{\rm Mov}}\, (M))^{\circ}$ of $\overline{{\rm Mov}}\, (M)$ coincides with the ${\rm Bir}\,(M)$-orbit of the nef cone $\overline{{\rm Amp}}\, (M)$:
$$\overline{{\rm Mov}}\, (M)^{\circ} = {\rm Bir}\, (M)^*\overline{{\rm Amp}}\, (M) := \cup_{h \in {\rm Bir}\, (M)} h^*\overline{{\rm Amp}}\, (M)\,\, .$$ 
\end{lemma}

\begin{proof}
Recall that ${\rm Aut}\, (M)$ preserves $\overline{{\rm Amp}}\, (M) = \R_{\ge 0} h_1 + \R_{\ge 0}h_2$, in particular, 
$$g^*\{h_1, h_2\} = \{h_1, h_2\}$$
 if $g \in {\rm Aut}\, (M)$ and ${\rm Bir}\, (M)$ preserves $\overline{{\rm Mov}}\, (M)$. As every nef divisor on $M$ is semi-ample in our situation, we have  
$${\rm Bir}\, (M)^*\overline{{\rm Amp}}\, (M) \subset {\overline{\rm Mov}}\, (M)\,\, .$$
By using the formula (\ref{eq1}), we also find that 
$$V^{\circ} = {\rm Bir}\, (M)^*\overline{{\rm Amp}}\, (M)$$ 
and therefore 
\begin{equation}\label{eq2}
V \subset {\overline{\rm Mov}}\, (M)\,\, .
\end{equation} 
Let $d \in {\overline{\rm Mov}}\, (M)^{\circ}$. Then $d$ is represented by an effective $\R$-divisor on $M$, say $D$. By Kodaira's lemma, $D$ is big. Choose a small positive real number $\epsilon > 0$ such that 
$$(M, \epsilon D) = (M, K_M + \epsilon D)$$
is klt. As $D$ is big, we can run the minimal model program for $(M, \epsilon D)$ to make $D$ nef by \cite[Theorem 1.2]{BCHM10}. By Lemma \ref{lem33}, $M$ has no divisorial contraction and all log-flips of $(M, \epsilon D)$ are $\tau_i$ ($i=1$, $2$), as all log-flips of $(M, \epsilon D)$ are necessarily flops of $M$ by $K_M = 0$. Thus, there is $g \in {\rm Bir}\, (M)$ such that $g^*D \in \overline{{\rm Amp}}\, (M)$. 
Therefore
\begin{equation}\label{eq3}
{\overline{\rm Mov}}\, (M)^{\circ} \subset {\rm Bir}\, (M)^*\overline{{\rm Amp}}\, (M) = V^{\circ}\,\, .
\end{equation}
As both $V$ and ${\overline{\rm Mov}}\, (M)$ are closed convex cone, the inclusions (\ref{eq2}) and (\ref{eq3}) imply the result. 
\end{proof}

We are now ready to complete the proof of Theorem \ref{thm31}. 

By Lemma \ref{lem31}, we see that $d_1(f) = (2n^2 -1) + 2n\sqrt{n^2 -1} >1$ 
and $f^*|N^1(M)$ is irreducible over $\Q$. Let $D \not= 0$ be a movable divisor on $M$. Then, as the class of $D$ is rational and the both boundary rays of $\overline{{\rm Mov}}\, (M)$ is irrational by the first part of Lemma \ref{lem34}, the class of $D$ belongs to $\overline{{\rm Mov}}\, (M)^{\circ}$. Thus by the second part of Lemma \ref{lem34}, there is $g \in {\rm Bir}\, (M)$ such that $g^*D \in \overline{{\rm Amp}}\, (M)$. As remarked at the beginning of this section, every nef divisor on $M$ is semi-ample. In particular, $g^*D$ is semi-ample. Now, we can apply Theorem \ref{thm1} to conclude. 
\end{proof}

\section*{Acknowledgements} I would like to express my thanks to Professors Ekaterina Amerik, Ljudmila Kamenova and De-Qi Zhang for their interest in this work and useful discussions. I would like to express my thanks to the referee for his/her constructive suggestions for explanations.

\end{document}